\documentclass[11pt]{article}
\usepackage{amsmath,amsthm,amssymb}
\usepackage[hyperfootnotes=false]{hyperref}
\usepackage{color}

\usepackage[symbol]{footmisc}
\allowdisplaybreaks

\usepackage{titlesec}
\titleformat{\paragraph}
{\normalfont\normalsize\bfseries}{\theparagraph}{1em}{}
\titlespacing*{\paragraph}
{0pt}{3.25ex plus 1ex minus .2ex}{1.5ex plus .2ex}

\def\titlerunning#1{\gdef\titrun{#1}}
\makeatletter
\def\author#1{\gdef\autrun{\def\and{\unskip, }#1}\gdef\@author{#1}}
\def\address#1{{\def\and{\\\hspace*{18pt}}\renewcommand{\thefootnote}{}%
\footnote {#1}}%
\markboth{\autrun}{\titrun}}
\makeatother
\def\email#1{\hspace*{4pt}{\em e-mail}: #1}
\def\MSC#1{{\renewcommand{\thefootnote}{}%
\footnote{\emph{Mathematics Subject Classification (2020):} #1}}}
\def\keywords#1{\par\medskip
\noindent\textbf{Keywords:} #1}

\newtheorem{theorem}{Theorem}[section]
\newtheorem{prop}[theorem]{Proposition}
\newtheorem{cor}[theorem]{Corollary}
\newtheorem{lemma}[theorem]{Lemma}

\theoremstyle{definition}

\newtheorem{cons}[theorem]{Construction}

\numberwithin{equation}{section}

\def\cL{\mathcal L}

\def\cA{\mathcal A}
\def\cC{\mathcal C}

\def\cD{\mathcal D}

\def\cG{\mathcal G}
\def\cH{\mathcal H}

\def\cP{\mathcal P}

\def\cU{\mathcal U}

\def\cT{\mathcal T}

\def\PG{{\rm PG}}

\def\GF{{\rm GF}}

\def\PGL{{\rm PGL}}

\def\PGU{{\rm PGU}}

\def\NU{{\rm NU}}

\begin{document}


\baselineskip=16pt

\titlerunning{}

\title{Graphs cospectral with $\NU(n+1, q^2)$, $n \ne 3$}

\author{Ferdinand Ihringer
\and
Francesco Pavese
\and
Valentino Smaldore}

\date{}

\maketitle

\address{F. Ihringer: Department of Mathematics: Analysis, Logic and Discrete Mathematics, Ghent University, Belgium \email{ferdinand.ihringer@ugent.be}
\and
F. Pavese: Dipartimento di Meccanica, Matematica e Management, Politecnico di Bari, Via Orabona 4, 70125 Bari, Italy \email{francesco.pavese@poliba.it}
\and
V. Smaldore: Dipartimento di Matematica, Informatica ed Economia, Universit{\`a} degli Studi della Basilicata, Contrada Macchia Romana, 85100, Potenza, Italy; \email{valentino.smaldore@unibas.it}
}

\bigskip

\MSC{Primary 51E20, 05E30; Secondary 05C50.}


\begin{abstract}
Let $\cH(n, q^2)$ be a non--degenerate Hermitian variety of $\PG(n, q^2)$, $n \ge 2$. Let $\NU(n+1, q^2)$ be the graph whose vertices are the points of $\PG(n, q^2) \setminus \cH(n, q^2)$ and two vertices $P_1, P_2$ are adjacent if the line joining $P_1$ and $P_2$ is tangent to $\cH(n, q^2)$. Then $\NU(n + 1, q^2)$ is a strongly regular graph. In this paper we show that $\NU(n+1, q^2)$, $n \ne 3$, is not determined by its spectrum. 

\end{abstract}

\section{Introduction}

A {\em strongly regular graph} with parameters $(v, k, \lambda, \mu)$ is a (simple, undirected, connected) graph with $v$ vertices such that each vertex lies on $k$ edges, any two adjacent vertices have exactly $\lambda$ common neighbours, and any two non--adjacent vertices have exactly $\mu$ common neighbours. Let $\cH(n,q^2)$ be a non--degenerate Hermitian variety of $\PG(n, q^2)$, $n \ge 2$. Let $\NU(n+1,q^2)$ be the graph whose vertices are the points of $\PG(n,q^2) \setminus \cH(n,q^2)$ and two vertices $P_1, P_2$ are adjacent if the line joining $P_1$ and $P_2$ contains exactly one point of $\cH(n,q^2)$ (i.e., it is a line tangent to $\cH(n,q^2)$). The graph $\NU(n+1,q^2)$ is a strongly regular graph for any $q$ \cite[Table 9.9, p. 145]{BH}. A variation of Godsil--McKay switching was described by Wang, Qiu, and Hu in \cite{WQH}. In \cite{IM}, by using a simplified version of the Wang--Qiu--Hu switching, the authors exhibited a strongly regular graph having the same parameters as $\NU(n+1, 4)$, $n \ge 5$, but not isomorphic to it. Moreover they asked whether or not the graph $\NU(n+1, q^2)$, is determined by its spectrum. In this paper we show that the graph $\NU(n + 1, q^2)$, $n \ne 3$, is not determined by its spectrum. 

A {\em unital} $\cU$ of a finite projective plane $\Pi$ of order $q^2$ is a set of $q^3+1$ points such that every line of $\Pi$ meets $\cU$ in either $1$ or $q+1$ points. A Hermitian curve is a unital of $\PG(2, q^2)$, called {\em classical unital}. Besides the classical one, there are other known unitals in $\PG(2, q^2)$ and they arise from a construction due to Buekenhout, see \cite{BE, B, M}. Let $\cU$ be a unital of $\PG(2, q^2)$ and let $\Gamma_{\cU}$ be the graph whose vertices are the points of $\PG(2, q^2) \setminus \cU$ and two vertices $P_1, P_2$ are adjacent if the line joining $P_1$ and $P_2$ contains exactly one point of $\cU$. The graph $\Gamma_{\cU}$ is strongly regular with the same parameters as $\NU(3, q^2)$. Here we observe the (well-known) fact that if $\cU$ is a non--classical unital, then $\Gamma_{\cU}$ is not isomorphic to $\NU(3, q^2)$.

In $\PG(n, q^2)$, $n \ge 4$, by applying a simplified version of the Wang--Qiu--Hu switching to the graph $\NU(n+1, q^2)$, we obtain two graphs $\cG'_n$ and $\cG''_n$, which are strongly regular and have the same parameters as $\NU(n+1, q^2)$. Moreover we show that $\cG'_n$ (if $q > 2$) and $\cG''_n$ are not isomorphic to $\NU(n+1, q^2)$.  

\section{Preliminary results}

Let $\PG(n, q^2)$ be the $n$--dimensional projective space over the finite $\GF(q^2)$ equipped with homogeneous projective coordinates $(X_1, \dots, X_{n+1})$. A unitary polarity of $\PG(n, q^2)$ is induced by a non--degenerate Hermitian form on the underlying vector space. A non--degenerate Hermitian variety $\cH(n, q^2)$ consists of the absolute points of a unitary polarity of $\PG(n, q^2)$. A non--degenerate Hermitian variety $\cH(n, q^2)$ has $\left(q^{n+1}+(-1)^n\right)\left(q^n-(-1)^n\right)/\left(q^2-1\right)$ points. Subspaces of maximal dimension contained in $\cH(n, q^2)$ are $\lfloor \frac{n-1}{2} \rfloor$--dimensional projective spaces and are called {\em generators}. There are either $(q + 1)(q^3 + 1) \dots (q^n + 1)$ or $(q^3 + 1)(q^5 + 1) \dots (q^{n+1} + 1)$ generators in $\cH(n, q^2)$, depending on whether $n$ is odd or even, respectively.  A line of $\PG(n, q^2)$ meets $\cH(n, q^2)$ in $1$, $q+1$ or, if $n \ge 3$, in $q^2 + 1$ points. The latter lines are the generators of $\cH(n, q^2)$ if $n \in \{3, 4\}$; lines meeting $\cH(n, q^2)$ in one or $q + 1$ points are called {\em tangent lines} or {\em secant lines}, respectively. Through a point $P$ of $\cH(n, q^2)$ there pass $\left(q^{n-1}+(-1)^{n-2}\right)\left(q^{n-2}-(-1)^{n-2}\right)/\left(q^2-1\right)$ generators and these generators are contained in a hyperplane. The hyperplane containing these generators is the polar hyperplane of $P$ with respect to the unitary polarity of $\PG(n, q^2)$ defining $\cH(n, q^2)$ and it is also called the {\em tangent hyperplane} to $\cH(n, q^2)$ at $P$. The tangent lines through $P$ are precisely the remaining lines of its polar hyperplane that are incident with $P$. If $P \notin \cH(n,q^2)$ then the polar hyperplane of $P$ is a hyperplane of $\PG(n, q^2)$ meeting $\cH(n, q^2)$ in a non--degenerate Hermitian variety $\cH(n-1, q^2)$ and it is said to be {\em secant} to $\cH(n, q^2)$. The stabilizer of $\cH(n, q^2)$ in $\PGL(n+1, q^2)$ is the group $\PGU(n+1, q^2)$. 

For a point $P$ of $\PG(2, q^2)$, its polar line meets $\cH(2, q^2)$ exactly in $P$ or in the $q+1$ points of a Baer subline, according as $P$ belongs to $\cH(2, q^2)$ or not. It is well known that $\cH(2, q^2)$ contains no {\em O’Nan configuration} \cite[p. 507]{O}, which is a configuration consisting of four distinct lines intersecting in six distinct points of $\cH(2, q^2)$. Dually, $\cH(2, q^2)$ cannot contain a {\em dual O'Nan configuration}, that is a configuration formed by four points of $\PG(2, q^2) \setminus \cH(2, q^2)$ no three on a line, such that the six lines connecting two of them are tangent lines. In $\PG(2, q^2)$, a {\em Hermitian pencil of lines} is a cone having as vertex a point $V$ and as base a Baer subline of a line $\ell$, where $V \notin \ell$. See \cite{BE} for more details.

A plane of $\PG(4, q^2)$ meets $\cH(4, q^2)$ in a line, a Hermitian pencil of lines or a non--degenerate Hermitian curve, according as its polar line has $q^2+1$, $1$ or $q+1$ points in common with $\cH(4, q^2)$, respectively.

\begin{lemma}\label{HermCurve1}
In $\PG(2, q^2)$, let $P_1, P_2, P_3$ be three points on a line $\gamma$ tangent to a non--degenerate Hermitian curve $\cH(2, q^2)$ at the point $T$, where $T \ne P_i$, $i = 1,2,3$. Let $s$ be the unique Baer subline of $\gamma$ containing $P_1, P_2, P_3$.
\begin{itemize}
\item If $T \in s$, then there is no point $P \in \PG(2, q^2) \setminus \left(\cH(2, q^2) \cup \gamma \right)$ such that the lines $P P_i$, $i= 1,2,3$, are tangent to $\cH(2, q^2)$.
\item If $T \notin s$, then there are $q$ points $P \in \PG(2, q^2) \setminus \left(\cH(2, q^2) \cup \gamma \right)$ such that the lines $P P_i$, $i= 1,2,3$, are tangent to $\cH(2, q^2)$.
\end{itemize}
\end{lemma}
\begin{proof}
Let $\perp$ be the unitary polarity of $\PG(2, q^2)$ defining $\cH(2, q^2)$. Assume first that $T \in s$. Suppose by contradiction that there is a point $P \in \PG(2, q^2) \setminus \left( \cH(2, q^2) \cup \gamma \right)$ such that the lines $P P_i$, $i= 1,2,3$, are tangent to $\cH(2, q^2)$. Then by projecting the $q+1$ points of $P^{\perp} \cap \cH(2, q^2)$ from $P$ onto $\gamma$ we obtain a Baer subline of $\gamma$ containing $P_1, P_2, P_3$, i.e., the Baer subline $s$. Since $T \in s$, we have that the line $P T$ has to be tangent to $\cH(2, q^2)$ at $T$. Hence $P T = \gamma$, contradicting the fact that $P \notin \gamma$.

Assume now that $T \notin s$. By using \cite[Theorem 15.3.11]{H1} is can be seen that there are $q^2(q-1)$ Baer sublines in $\gamma$ such that $T \notin s$, and these are permuted in a single orbit by the stabilizer $\PGU(3, q^2)_\gamma$ of $\gamma$ in $\PGU(3, q^2)$. On the other hand, $\PGU(3, q^2)_\gamma$ permutes in a unique orbit the $q^3(q-1)$ points of $\PG(2, q^2) \setminus \left( \cH(2, q^2) \cup \gamma \right)$, see for instance \cite[Lemma 2.7]{O}. Consider the incidence structure whose point set $\cP$ consists of the $q^4-q^3$ points of $\PG(2, q^2) \setminus \left( \cH(2, q^2) \cup \gamma \right)$ and whose block set $\cL$ of the $q^3-q^2$ Baer sublines of $\gamma$, where a point $P \in \cP$ is incident with a block $r \in \cL$ if $r$ is obtained by projecting the $q+1$ points of $P^{\perp} \cap \cH(2, q^2)$ from $P$ onto $\gamma$. Since both, $\cP$ and $\cL$, are orbits of the same group, the considered incidence structure is a tactical configuration \cite[Lemma 1.45]{HP} and hence we have that a block of $\cL$ is incident with $q$ points of $\cP$. In other words, there are $q$ points of $\PG(2, q^2) \setminus \left( \cH(2, q^2) \cup \gamma \right)$ that are projected onto the same Baer subline of $\gamma$, as required.
\end{proof}

\begin{lemma} \label{HermCurve0}
In $\PG(2, q^2)$, let $t, t'$ be two lines tangent to a non--degenerate Hermitian curve $\cH(2, q^2)$ and let $R = t \cap t'$. There is a bijection between the points $P$ of $t' \setminus \left( \cH(2, q^2) \cup \{R\} \right)$ and the Baer sublines $s_{P}$ of $t$ passing through $R$, not containing $\cH(2, q^2) \cap t$ and such that a line joining $P$ with a point of $s_P$ is tangent to $\cH(2, q^2)$.
\end{lemma}
\begin{proof}
It is enough to note that the stabilizer of $t$ and $t'$ in $\PGU(3, q^2)$ acts transitively on the $q^2-1$ points of $t' \setminus \left( \cH(2, q^2) \cup \{R\} \right)$ and the $q^2-1$ Baer sublines of $t$ through $R$ not containing $\cH(2, q^2) \cap t$.
\end{proof}

\begin{lemma} \label{HermCurve-1}
In $\PG(2, q^2)$, let $t$ be a line tangent to a non--degenerate Hermitian curve $\cH(2, q^2)$ at $T$ and let $u, u_1, u_2$ be three distinct points of $t$ such that the Baer subline of $t$ containing them does not pass through $T$. For a point $u' \in \PG(2, q^2) \setminus \cH(2, q^2)$ such that $\langle u, u' \rangle$ is secant to $\cH(2, q^2)$ and $\langle u', u_1 \rangle$, $\langle u', u_2 \rangle$ are tangents, there are exactly two points $R_i$, $i = 1, 2$, of $\PG(2, q^2) \setminus \left( \cH(2, q^2) \cup t \right)$ such that the lines $\langle R_i, u \rangle, \langle R_i, u_1 \rangle, \langle R_i, u_2 \rangle, \langle R_i, u' \rangle$ are tangent to $\cH(2, q^2)$. Moreover $\langle R_i, u' \rangle \cap t = u_i$.
\end{lemma}
\begin{proof}
We may assume without loss of generality that $\cH(2, q^2): X_1^q X_2 + X_1 X_2^q + X_3^{q+1} = 0$, $T = (1, 0, 0)$, $u = (0, 0, 1)$, $u_i = (x_i, 0, 1)$, with $x_1^q x_2 - x_1 x_2^q \ne 0$, otherwise $T$ would belong to the Baer subline of $t$ containing $u, u_1, u_2$. Let $R = (c, 1, d)$. By \cite[Lemma 2.3]{HT}, the lines $\langle R, u \rangle$, $\langle R, u_i \rangle$, $i = 1,2$, are tangents if and only if $c + c^q = 0$ and $d = \frac{x_1^{q+1} x_2^q - x_1^q x_2^{q+1}}{x_1^q x_2 - x_1 x_2^q}$. Let $u' = (a, 1, b)$, where $a + a^q \ne 0$, since $\langle u, u' \rangle$ is secant, $a+a^q+b^{q+1} \ne 0$, since $u' \notin \cH(2, q^2)$ and $b - d = \frac{(a+a^q) (x_1^q - x_2^q)}{x_1^q x_2 - x_1 x_2^q}$, since $|\langle u', u_i \rangle \cap \cH(2, q^2)| = 1$, $i = 1, 2$. Thus the line $\langle R, u' \rangle$ is tangent if and only if $c$ satisfies $X^2 + X (a^q - a + b^qd - bd^q) + ad (b - d)^q + a^qd^q (b - d) - a^{q+1} = 0$. The two solutions of the previous quadratic equation are $\frac{a x_1^q (x_2 - x_1) + a^q x_1 (x_2^q - x_1^q)}{x_1^q x_2 - x_1 x_2^q}$ and $\frac{a x_2^q (x_2 - x_1) + a^q x_2 (x_2^q - x_1^q)}{x_1^q x_2 - x_1 x_2^q}$. Moreover it can be easily checked that $\langle u', R \rangle \cap t \in \{ u_1, u_2 \}$.    
\end{proof}

\begin{lemma}\label{HermCurve2}
Let $\ell$ be a line of $\PG(2, q^2)$ and let $P_1, P_2, P_3$ be three non--collinear points of $\PG(2, q^2) \setminus \ell$. Then there are $q^2+2q$ Hermitian pencils of lines meeting $\ell$ in one point and containing $P_1, P_2, P_3$.
\end{lemma}
\begin{proof}
Let $\cC$ be a Hermitian pencil of lines consisting of $q+1$ lines through a point of $\ell$ and containing the points $P_1, P_2, P_3$. If the lines $P_1 P_2, P_1 P_3, P_2 P_3$ are not lines of $\cC$, then there are $q^2-2$ of such curves. Note that by projecting the Baer subline of $P_2 P_3$ containing $\ell \cap (P_2 P_3), P_2, P_3$ from $P_1$ onto $\ell$ and by projecting the Baer subline of $P_1 P_3$ containing $\ell \cap (P_1 P_3), P_1, P_3$ from $P_2$ onto $\ell$ and by projecting the Baer subline of $P_1 P_2$ containing $\ell \cap (P_1 P_2), P_1, P_2$ from $P_3$ onto $\ell$, we obtain the same Baer subline $s$ of $\ell$. If a Hermitian pencil of lines contains $\ell$ and the points $P_1, P_2, P_3$, but does not contain the lines $P_1 P_2, P_1 P_3, P_2 P_3$, then  its vertex is a point of $s \setminus \left\{ \ell \cap (P_1 P_2), \ell \cap (P_1 P_3), \ell \cap (P_2 P_3) \right\}$. Hence, among the $q^2-2$ Hermitian pencils of lines described above there are exactly $q-2$ that contain $\ell$. There are $q+1$ Hermitian pencils of lines consisting of $q+1$ lines through the point of $\ell \cap (P_1 P_2)$, containing the points $P_1, P_2, P_3$ and the line $P_1 P_2$. Exactly one of these contains $\ell$. Similarly for $P_1 P_3$ and $P_2 P_3$. Therefore there are $q^2-2 - (q-2) + 3(q+1) - 3 = q^2+2q$ Hermitian pencils of lines meeting $\ell$ in one point and containing $P_1, P_2, P_3$.
\end{proof}

\begin{lemma}\label{HermCurve3}
In $\PG(2, q^2)$, let $\cH(2, q^2)$ be a non--degenerate Hermitian curve and let $P_1, P_2, P_3$ be three non--collinear points of $\PG(2, q^2) \setminus \cH(2, q^2)$ such that the line $P_i P_j$ is tangent to $\cH(2, q^2)$, $1 \le i < j \le 3$. Then there are exactly $3q$ Hermitian pencils of lines meeting $\cH(2, q^2)$ in $q+1$ points and containing $P_1, P_2, P_3$.
\end{lemma}
\begin{proof}
Let $\perp$ be the unitary polarity of $\PG(2, q^2)$ defining $\cH(2, q^2)$. Let $\cC$ be a Hermitian pencil of lines such that $|\cC \cap \cH(2, q^2)| = q+1$ and $P_i \in \cC$, $i = 1,2,3$. Let $V$ be the vertex of $\cC$. Then necessarily $V^{\perp} = \langle \cC \cap \cH(2, q^2) \rangle$ and each of the $q+1$ lines of $\cC$ is tangent to $\cH(2, q^2)$. If $V$ is one among $P_1, P_2, P_3$, then $\cC$ is uniquely determined. If $V \notin \{P_1, P_2, P_3\}$, then $V$ must lie on one of the lines $P_1 P_2, P_1 P_3, P_2 P_3$, otherwise $\cH(2, q^2)$ would contain the dual of an O'Nan configuration, a contradiction. On the other hand the point $V$ can be chosen in $q-1$ ways on each of the three lines $P_1 P_2, P_1 P_3, P_2 P_3$ and for every choice of $V$ the curve $\cC$ is uniquely determined.
\end{proof}

\begin{lemma}\label{HermCurve4}
In $\PG(2, q^2)$, let $\cH$ and $\cH'$ be two non--degenerate Hermitian curves with associated polarity $\perp$ and $\perp'$, respectively. Let $|\cH \cap \cH'| \in \{1, q+1\}$ and let $\ell$ be the (unique) line of $\PG(2, q^2)$ such that $\ell \cap \cH = \ell \cap \cH' = \cH \cap \cH'$. Thus $\ell^\perp = \ell^{\perp'}$ and if $P \in \cH \setminus \cH'$, then $P^{\perp'} \cap P^\perp \in \ell$.
\end{lemma}
\begin{proof}
We may assume without loss of generality that $\cH$ is given by $X_1^qX_2 + X_1X_2^q + X_3^{q+1} = 0$ and that $\cH'$ is given by $X_1^qX_2 + X_1X_2^q + \lambda X_3^{q+1} = 0$, $\lambda \in \GF(q) \setminus \{0, 1\}$, if $|\cH \cap \cH'| = q+1$, or that $\cH'$ is given by $\lambda X_1^{q+1} + X_1^qX_2 + X_1X_2^q + X_3^{q+1} = 0$, $\lambda \in \GF(q) \setminus \{0\}$, if $|\cH \cap \cH'| = 1$. For more details on the intersection of two non--degenerate Hermitian curves see \cite{BE, G, K}. In the former case $\ell$ is the line $X_3 = 0$ and $\ell^\perp = \ell^{\perp'} = (0, 0, 1)$. Let $P = (a, b, 1) \in \cH \setminus \cH'$, then $P^\perp \cap P^{\perp'} = (a^q, -b^q, 0) \in \ell$. In the latter case $\ell$ is the line $X_1 = 0$ and $\ell^\perp = \ell^{\perp'} = (0, 1, 0)$. Let $P = (1, a, b) \in \cH \setminus \cH'$, then $P^\perp \cap P^{\perp'} = (0, -b^q, 1) \in \ell$.
\end{proof}

\begin{prop}\label{HermCurve5}
In $\PG(2, q^2)$, let $\cH(2, q^2)$ be a non--degenerate Hermitian curve and let $P_1, P_2, P_3$ be three non--collinear points of $\PG(2, q^2) \setminus \cH(2, q^2)$ such that the line $P_i P_j$ is tangent to $\cH(2, q^2)$, $1 \le i < j \le 3$. Then there are exactly $q^2-q+1$ non--degenerate Hermitian curves meeting $\cH(2, q^2)$ in one or $q+1$ points and containing $P_1, P_2, P_3$.
\end{prop}
\begin{proof}
Assume that $P_1 = (1,0,0)$, $P_2 = (0,1,0)$, $P_3 = (0,0,1)$. If $\cH$ is a non--degenerate Hermitian curve containing $P_1, P_2, P_3$, then $\cH$ is given by
$$
a_{12} X_1 X_2^q + a_{12}^q X_1^qX_2 + a_{13} X_1X_3^q + a_{13}^q X_1^qX_3 + a_{23} X_2X_3^q + a_{23}^q X_2^qX_3 = 0,
$$
for some $a_{12}, a_{13}, a_{23} \in \GF(q^2)$, with
\begin{equation} \label{cond1}
a_{12} a_{13}^q a_{23} + a_{12}^q a_{13} a_{23}^q \ne 0.
\end{equation}
On the other hand, if $\cH(2, q^2)$ is a non--degenerate Hermitian curve such that $P_1, P_2, P_3 \in \PG(2, q^2) \setminus \cH(2, q^2)$ and the line $P_i P_j$ is tangent to $\cH(2, q^2)$, $1 \le i < j \le 3$, then $\cH(2, q^2)$ is given by
$$
y^{q+1} X_1^{q+1} + z^{q+1} X_2^{q+1} + X_3^{q+1} - xz^{q+1} X_1X_2^q - x^qz^{q+1} X_1^qX_2 - y X_1X_3^q - y^q X_1^qX_3 - z X_2X_3^q - z^q X_2^qX_3 = 0,
$$
for some $x, y, z \in \GF(q^2)$ such that $y^{q+1} = x^{q+1} z^{q+1}$ and $y \ne - x z$. By using the projectivity $X'_1 = y X_1, X'_2 = zX_2, X'_3 = X_3$, we may assume that $\cH(2, q^2)$ is given by
$$
X_1^{q+1} + X_2^{q+1} + X_3^{q+1} - \alpha X_1X_2^q - \alpha^q X_1^qX_2 - X_1X_3^q - X_1^qX_3 - X_2X_3^q - X_2^qX_3 = 0,
$$
for some $\alpha \in \GF(q^2)$ such that $\alpha^{q+1} = 1$ and $\alpha \ne -1$. Then the lines $P_1 P_2$, $P_1 P_3$ and $P_2 P_3$ are tangent to $\cH(2, q^2)$ at the points $L_3 = (1, \alpha, 0)$, $L_2 = (1, 0, 1)$ and $L_1 = (0, 1, 1)$, respectively. Let $\perp$ and $\perp'$ be the polarities defining $\cH(2, q^2)$ and $\cH$, respectively. If $|\cH(2, q^2) \cap \cH| \in \{1, q+1\}$, according to Lemma \ref{HermCurve4}, we have that $L_i^\perp \cap L_i^{\perp'} \in \ell$, $i = 1,2,3$, where $\ell$ is the unique line of $\PG(2, q^2)$ such that $\ell \cap \cH(2, q^2) = \ell \cap \cH = \cH(2, q^2) \cap \cH$. Denote by $Z_i$ the point $L_i^\perp \cap L_i^{\perp'}$, $i = 1,2,3$ and let $L$ be the point $\ell^\perp$. Then
$$
Z_1 = (0, -a_{23}^q, a_{23}), Z_2 = (-a_{13}^q, 0, a_{13}), Z_3 = (-a_{12}^q, \alpha^q a_{12}, 0)
$$
are on a line if and only if
\begin{equation}\label{cond2}
a_{12}^q a_{13} a_{23}^q - \alpha^q a_{12} a_{13}^q a_{23} = 0.
\end{equation}
Note that $L = Z_1^\perp \cap Z_2^\perp = (a_{13}(a_{23}+a_{23}^q), \alpha a_{23} (a_{13}+a_{13}^q), \alpha a_{13}^q a_{23} + a_{13} a_{23}^q)$ and since $\ell^\perp = \ell^{\perp'}$, we have that $Z_1, Z_2 \in L^{\perp'}$. Some calculations show that $Z_1, Z_2 \in L^{\perp'}$ if and only if
\begin{align}
& (a_{23} + a_{23}^q)(a_{23} a_{13}^q - a_{23}^q a_{12}^q) + a_{23}^{q+1}(\alpha^q a_{23}^q - a_{23}) = 0, \label{cond3}\\
& (a_{13} + a_{13}^q)(a_{23}^q a_{13} - a_{13}^q a_{12}) + a_{13}^{q+1}(\alpha a_{13}^q - a_{13}) = 0. \label{cond4}
\end{align}
Since $\alpha \ne -1$, it follows from \eqref{cond4} that
\begin{equation} \label{cond0}
a_{13} + a_{13}^q \ne 0.
\end{equation}
Hence from \eqref{cond4}, we have that
\begin{equation}\label{coeff1}
a_{12} = \frac{a_{13} a_{23}^q \left( a_{13} + a_{13}^q \right) + a_{13}^{q+1} \left( \alpha a_{13}^q - a_{13} \right)}{a_{13}^q \left( a_{13}+a_{13}^q \right)}.
\end{equation}
By substituting \eqref{coeff1} in \eqref{cond2}, we get that one of the following has to be satisfied:
\begin{align}
& a_{13}^q - \alpha^q a_{13} = 0, \label{cond5}\\
& a_{23}^{q+1}(a_{13} + a_{13}^q) - a_{13}^{q+1}(a_{23} + a_{23}^q) = 0. \label{cond6}
\end{align}

Assume that \eqref{cond6} is satisfied. Since $a_{23} + a_{23} \ne 0$, then
\begin{equation}\label{coeff2}
a_{13} = \frac{a_{23}^{q+1}}{a_{23} + a_{23}^q} \left( \xi + 1 \right),
\end{equation}
where $\xi \in \GF(q^2), \xi^{q+1} = 1$, with $\xi \ne -1$, otherwise $a_{13} = 0$, contradicting \eqref{cond1}. By taking into account \eqref{coeff2}, equation \eqref{coeff1} becomes
\begin{equation}\label{coeff1_new}
a_{12} = \frac{a_{23}^q\left( \xi a_{23}^q + \alpha a_{23} \right)}{a_{23} + a_{23}^q}.
\end{equation}
Moreover, if \eqref{coeff2} and \eqref{coeff1_new} hold true, then \eqref{cond3} is satisfied, whereas some calculations show that condition \eqref{cond1} is satisfied whenever $\xi a_{23}^q + \alpha a_{23} \ne 0$, i.e,
\begin{equation} \label{cond7}
a_{23}^{q-1} \ne -\xi^q \alpha.
\end{equation}
The equation $X^{q-1} = -\xi^q \alpha$ has $q-1$ solutions in $\GF(q^2)$ since $\left(- \xi^q \alpha \right)^{q+1} = 1$.

Assume that \eqref{cond5} holds true. Thus it turns out that necessarily $a_{12} = \alpha a_{23}^q$ and $a_{13} = a_{23}^{q+1}(\alpha+1)/\left( a_{23} +a_{23}^q \right)$ and hence we retrieve a particular solution of the previous case (the one with $\xi = \alpha$).

Observe that for a fixed $\xi \in \GF(q^2)$, with $\xi^{q+1} = 1$, $\xi \ne -1$, we have that
$$
|\{a_{23} \in \GF(q^2) \;\; | \;\; a_{23} + a_{23}^q \ne 0, a_{23}^{q-1} \ne -\xi^q \alpha \}| =
\begin{cases}
q(q-1) & \mbox{ if } \xi = \alpha, \\
(q-1)^2 & \mbox{ if } \xi \ne \alpha.
\end{cases}
$$
Furthermore for a fixed $a_{23} \in \GF(q^2)$ satisfying \eqref{cond0} and \eqref{cond7}, then $a_{12}$ and $a_{13}$ are uniquely determined by \eqref{coeff1_new} and \eqref{coeff2}, respectively, and, if $\rho \in \GF(q) \setminus \{0\}$, the values $a_{23}, \rho a_{23}$ give rise to the same Hermitian curve.
\end{proof}

\begin{lemma}\label{HermSur}
In $\PG(3, q^2)$, let $\cH(3, q^2)$ be a non--degenerate Hermitian surface with associated polarity $\perp$ and let $\gamma = P^\perp$ be a plane secant to $\cH(3, q^2)$. If $P'$ is a point of $\PG(3, q^2) \setminus \left( \cH(3, q^2) \cup \gamma \cup \{P\} \right)$, then the points of $\gamma$ lying on a line passing through $P'$ and tangent to $\cH(3, q^2)$ form a non--degenerate Hermitian curve $\cH'$ of $\gamma$ such that $\gamma \cap \cH' \cap \cH(3, q^2) \subset \langle P P' \rangle^\perp$. Viceversa, if $\cH'$ is a non--degenerate Hermitian curve of $\gamma$ meeting $\gamma \cap \cH(3, q^2)$ in one or $q+1$ points and $\ell$ is the (unique) line of $\gamma$ such that $\ell \cap \cH' = \ell \cap \cH(3, q^2) = \gamma \cap \cH' \cap \cH(3, q^2)$, then there are exactly $q+1$ points $P'$ of $\ell^\perp \setminus \left( \cH(3, q^2) \cup \gamma \cup \{P\} \right)$ such that the lines joining $P'$ with a point of $\cH'$ are tangent to $\cH(3, q^2)$.
\end{lemma}
\begin{proof}
Let $\cH(3, q^2)$ be the Hermitian surface of $\PG(3, q^2)$ given by $X_1^{q+1}+X_2^{q+1}+X_3^{q+1}+X_4^{q+1} = 0$ and let $\gamma$ be the plane $X_4 = 0$. Let $P'$ be a point of $\PG(3, q^2) \setminus \left( \cH(3, q^2) \cup \gamma \cup \{P\} \right)$. If $PP'$ is tangent to $\cH(3, q^2)$, then we may assume that $P P' \cap \gamma = (1, \xi, 0, 0)$, for some $\xi \in \GF(q^2)$ such that $\xi^{q+1} = -1$. Then $P' = (1, \xi, 0, \lambda)$ where $\lambda \in \GF(q^2) \setminus \{0\}$. Similarly if $PP'$ is secant to $\cH(3, q^2)$, then we may assume that $P P' \cap \gamma = (1, 0, 0, 0)$. In this case $P' = (1, 0, 0, \lambda)$ where $\lambda \in \GF(q^2) \setminus \{0\}$ and $\lambda^{q+1} \ne -1$. From \cite[Lemma 2.3]{HT}, the line joining $P'$ and $Q = (x_1, x_2, x_3, 0) \in \gamma$ is tangent to $\cH(3, q^2)$ if and only if $Q \in \cH'$, where $\cH'$ is the non--degenerate Hermitian curve of $\gamma$ given by
$$
\left( \lambda^{q+1} - 1 \right) x_1^{q+1} + \left( \lambda^{q+1} + 1 \right) x_2^{q+1} + \lambda^{q+1} x_3^{q+1} - \xi^q x_1^q x_2 - \xi x_1 x_2^q = 0
$$
if $|PP' \cap \cH(3, q^2)| = 1$ or by
$$
\lambda^{q+1} x_1^{q+1} + \left( \lambda^{q+1} + 1 \right) \left( x_2^{q+1} + x_3^{q+1} \right) = 0
$$
if $|PP' \cap \cH(3, q^2)| = q+1$.
Note that $\gamma \cap \cH' \cap \cH(3, q^2) \subset \langle P P' \rangle^\perp$ and by replacing $\lambda$ with $\lambda'$, where $\lambda^{q+1} = \lambda'^{q+1}$ we retrieve the same Hermitian curve $\cH'$.
\end{proof}

\section{Graphs cospectral with $\NU(3, q^2)$}

Let $\Pi$ be a finite projective plane of order $q^2$ and let $\cU$ be a unital of $\Pi$. A line of $\Pi$ is said to be {\em tangent} or {\em secant} to $\cU$ according as it meets $\cU$ in $1$ or $q+1$ points, respectively. There are $q^3 + 1$ tangent lines and $q^4 - q^3 + q^2$ secant lines. In particular, through each point of $\cU$ there pass $q^2$ secant lines and one tangent line, while through each point of $\Pi \setminus \cU$ there pass $q + 1$ tangent lines and $q^2 - q$ secant lines. Starting from $\cU$, it is possible to define a unital $\cU^*$ of the dual plane $\Pi^*$ of $\Pi$. This can be done by taking the tangent lines to $\cU$ as the points of $\cU^*$ and the points of $\Pi \setminus \cU$ as the secant lines to $\cU^*$, where incidence is given by reverse containment. Thus $\cU^*$ is a unital of $\Pi^*$, called the {\em dual unital} to $\cU$. 

Let $\Gamma_{\cU}$ be the graph whose vertices are the points of $\Pi \setminus \cU$ and two vertices $P_1, P_2$ are adjacent if the line joining $P_1$ and $P_2$ is tangent to $\cU$.

\begin{prop}
The graph $\Gamma_{\cU}$ is strongly regular with parameters $(q^2(q^2-q+1), (q+1)(q^2-1), 2(q^2-1), (q+1)^2)$.
\end{prop}
\begin{proof}
Since every point of $\Pi \setminus \cU$ lies on $q+1$ lines that are tangent to $\cU$, the graph $\Gamma_{\cU}$ is $(q+1)(q^2-1)$--regular. Consider two distinct vertices $u_1, u_2$ of $\Gamma_{\cU}$ corresponding to the points $P_1, P_2$ of $\Pi \setminus \cU$. If $u_1, u_2$ are adjacent, then the line $P_1 P_2$ is tangent to $\cU$ and $u_1, u_2$ have $2(q^2-1)$ common neighbours. If $u_1, u_2$ are not adjacent, then the line $P_1 P_2$ is secant to $\cU$ and $u_1, u_2$ have $(q+1)^2$ common neighbours.
\end{proof}

A {\em clique} of a graph is a subset of pairwise adjacent vertices. A clique is said to be {\em maximal} if it is maximal with respect to set theoretical inclusion.

\begin{prop}\label{cliques}
The graph $\Gamma_{\cU}$ contains a class of $q^3+1$ maximal cliques of size $q^2$ and a class of $q^2(q^3+1)(q^2-q+1)$ maximal cliques of size $q+2$.
\end{prop}
\begin{proof}
In the first case, a clique corresponds to the points on a tangent line with the tangent point excluded. In the second case, it corresponds to $q+1$ points on a tangent line $\ell$, and the other point off $\ell$. In both cases these cliques are maximals.
\end{proof}
 
Let $\Pi = \PG(2, q^2)$ and let $\cU$ be the classical unital $\cH(2, q^2)$. In this case $\Gamma_{\cU}$ is denoted by $\NU(3, q^2)$. Since a dual O'Nan configuration cannot be embedded in $\cH(2, q^2)$, we have the following (see also \cite[Corollary 3]{BF}).

\begin{cor}\label{cor1}
The maximal cliques of the graph $\NU(3,q^2)$ are exactly those described in Proposition \ref{cliques}.
\end{cor}

In what follows we recall few known properties of some classes unitals of $\PG(2, q^2)$. Besides the classical one, there are other known unitals in $\PG(2, q^2)$ and all the known unitals of $\PG(2, q^2)$ arise from a construction due to Buekenhout, see \cite{BE, B, M}. A {\em Buekenhout unital} of $\PG(2, q^2)$ is either an orthogonal Buekenhout--Metz unital or a Buekenhout--Tits unital. An {\em orthogonal Buekenhout--Metz unital} of $\PG(2, q^2)$, $q > 2$, is projectively equivalent to one of the following form:
$$
\cU_{\alpha, \beta} = \{(x, \alpha x^2 + \beta x^{q+1} + z, 1) \;\; | \;\; z \in \GF(q), x \in \GF(q^2)\} \cup \{(0,1,0)\} ,
$$
where $\alpha, \beta$ are elements in $\GF(q^2)$ such that $(\beta - \beta^q)^2 + 4 \alpha^{q+1}$ is a non--square in $\GF(q)$ if $q$ is odd, or $\beta \notin \GF(q)$ and $\alpha^{q+1}/(\beta+\beta^q)^2$ has absolute trace $0$ if $q$ is even and $q > 2$. The unital $\cU_{\alpha, \beta}$ is classical if and only if $\alpha = 0$. If $q$ is odd, $\beta \in \GF(q)$ and $\alpha \ne 0$, then $\cU_{\alpha, \beta}$ can be obtained glueing together $q$ conics of $\PG(2,q^2)$ having in common the point $(0, 1, 0)$. Let $m > 1$ be an odd integer and let $q = 2^m$. A {\em Buekenhout--Tits unital} of $\PG(2, q^2)$ is projectively equivalent to
$$
\cU_{T} = \{(x_0 + x_1 \beta, (x_0^{\delta + 2} + x_0 x_1 + x_1^{\delta}) \beta + z, 1) \;\; | \;\; x_0, x_1, z \in \GF(q)\} \cup \{(0,1,0)\} ,
$$
where $\beta$ is an element of $\GF(q^2) \setminus \GF(q)$ and $\delta = 2^{\frac{m+1}{2}}$. See \cite{FL} for more details.

As the Desarguesian plane is self--dual, that is, it is isomorphic to its dual plane, the dual of a unital in $\PG(2, q^2)$ is a unital in $\PG(2, q^2)$. The non--classical orthogonal Buekenhout--Metz unital $\cU_{\alpha, \beta}$ as well as the Buekenhout--Tits unital $\cU_{T}$  of $\PG(2, q^2)$ are self--dual \cite[Theorem 4.17, Theorem 4.28, Corollary 4.35]{BE}.  Moreover the O'Nan configuration can be embedded in each non--classical Buekenhout unital of $\PG(2, q^2)$ \cite{FL}. Therefore it follows that there exists a dual O'Nan configuration embedded in both unitals $\cU_{\alpha, \beta}$ and $\cU_{T}$\footnote{We thank A. Aguglia for pointing this out.}. 

\begin{prop}
There exists a dual O'Nan configuration embedded in a non--classical Buekenhout unital of $\PG(2, q^2)$.
\end{prop}

For the convenience of the reader we illustrate the previous statement with an example. Let $q$ be odd and let us fix $\alpha, \beta$ in $\GF(q^2)$, $\alpha \ne 0$, such that $(\beta - \beta^q)^2 + 4 \alpha^{q+1}$ is a non--square in $\GF(q)$. Set $\cU = \{(-2 \alpha x + (\beta^q - \beta) x^q, 1, \alpha x^2 - \beta^q x^{q+1} - z) \;\; | \;\; x \in \GF(q^2), z \in \GF(q) \} \cup \{(0, 0, 1)\}$. From the proof of \cite[Theorem 4.17]{BE} $\cU$ is projectively equivalent to $\cU_{\alpha, \beta}$. Consider the six lines given by: $[0, 0, 1]$, $[0, -2 x_{\lambda_1} x_{\lambda_2}, x_{\lambda_1} + x_{\lambda_2}]$, $[x_{\lambda_1}, - x_{\lambda_1}, 1]$, $[x_{\lambda_2}, - x_{\lambda_2}, 1]$, $[- x_{\lambda_1}, - x_{\lambda_1}, 1]$, $[- x_{\lambda_2}, - x_{\lambda_2}, 1]$, where $x_{\lambda_i} = - \frac{\alpha^q + \lambda_i - \beta}{\alpha^{q+1} - (\lambda_i - \beta)^{q+1}}$, $i = 1,2$, and $\lambda_1, \lambda_2 \in \GF(q)$ are such that $(\lambda_1 + \alpha^q - \beta)^{q+1}(\alpha^{q+1} - \beta^{q+1} - \lambda_2^2) + (\lambda_2 + \alpha^q - \beta)^{q+1}(\alpha^{q+1} - \beta^{q+1} -\lambda_1^2) = 0$. By \cite[Theorem 4.17]{BE} and \cite[Corollary 3.6]{FL} these lines are tangent to $\cU$ and there are four non--collinear points of $\PG(2, q^2) \setminus \cU$ such that the lines joining two of these four points are exactly the six lines given above.  

An easy consequence of the previous proposition is the following result.   

\begin{cor}\label{cor2}
If $\cU$ is a non--classical Buekenhout unital of $\PG(2, q^2)$, then the graph $\Gamma_{\cU}$ has at least a further class of cliques than those described in Proposition \ref{cliques}. A clique in this class contains four points no three on a line.
\end{cor}

Comparing Corollary \ref{cor1} with Corollary \ref{cor2}, we see that $\Gamma_{\cU}$ has more maximal cliques than 
$\Gamma_{\cU}$ and we obtain the desired result for $q>2$.

\begin{theorem}
Let $q>2$. If $\cU$ is a non--classical Buekenhout unital of $\PG(2, q^2)$, then the graph $\Gamma_{\cU}$ is not isomorphic to $\NU(3, q^2)$.
\end{theorem}

\section{Graphs cospectral with $\NU(n+1, q^2)$, $n \ge 4$}

In this section we introduce two classes of graphs cospectral with $\NU(n+1, q^2)$, $n \ge 4$. These graphs are obtained by modifying few adjacencies in $\NU(n+1, q^2)$ following an idea proposed by Wang, Qiu, and Hu. For a graph $\cG$ and for a vertex $u$ of $\cG$, let $N_{\cG}(u)$ be the set of neighbours of $u$ in $\cG$. Let us recall the Wang--Qiu--Hu switching in a simplified version as stated in \cite{IM}.

\begin{theorem}[WQH switching]\label{thm:wqh}
Let $\cG$ be a graph whose vertex set is partitioned as $\ell_1 \cup \ell_2 \cup \cD$. Assume that the induced subgraphs on $\ell_1, \ell_2,$ and $\ell_1 \cup \ell_2$ are regular, and that the induced subgraphs on $\ell_1$ and $\ell_2$ have the same size and degree. Suppose that each $x \in \cD$ either has the same number of neighbours in $\ell_1$ and $\ell_2$, or $N_{\cG}(x) \cap (\ell_1 \cup \ell_2) \in \{ \ell_1, \ell_2 \}$. Construct a new graph $\cG'$ by switching adjacency and non--adjacency between $x \in \cD$ and $\ell_1 \cup \ell_2$ when $N_{\cG}(x) \cap (\ell_1 \cup \ell_2) \in \{ \ell_1, \ell_2 \}$. Then $\cG$ and $\cG'$ are cospectral.
\end{theorem}

\subsection{The graphs $\cG'_n$ and $\cG''_n$}

Let $\cH(n, q^2)$ be a non--degenerate Hermitian variety of $\PG(n, q^2)$, $n \ge 4$, with associated polarity $\perp$. Let us denote by $\cG_n$ the graph $\NU(n+1, q^2)$. Recall that the vertices of $\cG_n$ are the points of $\PG(n, q^2) \setminus \cH(n, q^2)$ and that two vertices are adjacent if their span is a line tangent to $\cH(n, q^2)$. The graph $\cG_n$ is strongly regular with parameters
\begin{align*}
& v = \frac{q^n(q^{n+1} - \epsilon)}{q+1}, \quad k = (q^n + \epsilon)(q^{n-1} - \epsilon), \\
& \lambda = q^{2n-3}(q+1) - \epsilon q^{n-1}(q-1) - 2, \quad \mu = q^{n-2}(q+1)(q^{n-1} - \epsilon),
\end{align*}
where $\epsilon = (-1)^{n+1}$.
Fix a point $P \in \cH(n, q^2)$ and let $\bar{\ell}_1, \bar{\ell}_2$ be two lines of $\PG(n, q^2)$ such that $\bar{\ell}_1 \cap \cH(n, q^2) = \bar{\ell}_2 \cap \cH(n, q^2) = P$. Let $\pi$ be the plane spanned by $\bar{\ell}_1$ and $\bar{\ell}_2$ and let $\ell_i = \bar{\ell}_i \setminus \{P\}$, $i = 1, 2$. Thus $\ell_i$ consists of $q^2$ vertices of $\cG_n$. There are two possibilities: either $\pi \cap \cH(n, q^2)$ is a Hermitian pencil of lines or is a line. Let us define the following sets:
\begin{align*}
& \cA = \left( \cap_{u \in \ell_1} N_{\cG_n}(u) \right) \cap \left( \cap_{u \in \ell_2} N_{\cG_n}(u) \right), \\
& \cA_1 = \left( \cap_{u \in \ell_1} N_{\cG_n}(u) \right) \setminus (\cA \cup \ell_2), \\
& \cA_2 = \left( \cap_{u \in \ell_2} N_{\cG_n}(u) \right) \setminus (\cA \cup \ell_1).
\end{align*}

\begin{lemma}\label{sets12}
Let $n = 4$. 
\begin{itemize}
\item If $\pi \cap \cH(4, q^2)$ is a Hermitian pencil of lines, then the set $\cA$ or $\cA_i$, $i= 1,2$, consists of the points lying on $(q+1)^2$ or $(q+1)(q^2-q-2)$ lines tangent to $\cH(4, q^2)$ at $P$, minus $P$ itself. In particular $|\cA| = q^2(q+1)^2$ and $|\cA_1| = |\cA_2| = q^2(q+1)(q^2-q-2)$.
\item If $\pi \cap \cH(4, q^2)$ is a line, then the set $\cA$ or $\cA_i$, $i= 1,2$, consists of the points lying on $2(q^2-1)$ or $q(q^2-q-1)$ lines tangent to $\cH(4, q^2)$ at $P$, minus $P$ itself. In particular $|\cA| = 2q^2(q^2-1)$ and $|\cA_1| = |\cA_2| = q^3(q^2-q-1)$.
\end{itemize}
\end{lemma}
\begin{proof}
Assume that $\pi \cap \cH(4, q^2)$ is a Hermitian pencil of lines. Let $u'$ be a vertex of $\cG_{4}$ such that it is adjacent to any vertex of $\ell_1$. Then $u' \notin \ell$, otherwise $u' \in N_{\cG_{4}}(u')$, a contradiction. If $u \in \ell_1$, the line determined by $u$ and $u'$ contains a unique point of $\cH(4, q^2)$ and hence the plane $\sigma$ spanned by $\bar{\ell}_1$ and $u'$ has exactly $q^2+1$ points of $\cH(4, q^2)$. Therefore $\sigma \cap \cH(4, q^2)$ is a line and $\sigma$ contains $q^4-q^2$ vertices of $\cG_{4}$ adjacent to every vertex of $\ell_1$. Note that there are $q+1$ planes through $\ell_1$ meeting $\cH(4, q^2)$ in a line. Hence $|\cA_1 \cup \cA| = (q+1)(q^4-q^2)$. A similar argument holds for $\ell_2$. Moreover if $\sigma_i$ is a plane through $\ell_i$ having $q^2+1$ points in common with $\cH(4, q^2)$, $i = 1,2$, then $\sigma_1 \cap \sigma_2$ is a line tangent to $\cH(4, q^2)$ at $P$. Therefore $\cA = q^2(q+1)^2$ and the first part of the statement follows. If $\pi \cap \cH(n, q^2)$ is a line, one can argues as before. 
\end{proof}

It can be easily deduced that $\cA, \cA_1, \cA_2 \subseteq P^\perp$. If $\pi \cap \cH(n, q^2)$ is a Hermitian pencil of lines and $(n, q) = (4, 2)$, then $|\cA_1| = |\cA_2| = 0$.

\begin{cons}\label{cons}
Define a graph $\cG'_n$ or $\cG''_n$, according as $\pi \cap \cH(n, q^2)$ is a Hermitian pencil of lines or a line. The vertices of $\cG'_n$ and $\cG''_n$ are the vertices of $\cG_n$ and the edges are as follows:
$$
\begin{cases}
N_{\cG_n}(u) & \mbox{ if } u \notin \cA_1 \cup \cA_2 \cup \ell_1 \cup \ell_2, \\
\left( N_{\cG_n}(u) \setminus \cA_1 \right) \cup \cA_2 & \mbox{ if } u \in \ell_1, \\
\left( N_{\cG_n}(u) \setminus \cA_2 \right) \cup \cA_1 & \mbox{ if } u \in \ell_2, \\
\left( N_{\cG_n}(u) \setminus \ell_1 \right) \cup \ell_2 & \mbox{ if } u \in \cA_1, \\
\left( N_{\cG_n}(u) \setminus \ell_2 \right) \cup \ell_1 & \mbox{ if } u \in \cA_2.
\end{cases}
$$
\end{cons}

Note that the graphs $\cG'_n$ and $\cG_n''$ are obtained from $\cG_n$ by applying the WQH switching as described in Theorem~\ref{thm:wqh}. Indeed, the vertex set of $\cG_n$ consists of $\ell_1 \cup \ell_2 \cup \cD$, where $\cA, \cA_1, \cA_2 \subset \cD$. Furthermore, if $x \in \cD \setminus (\cA_1 \cup \cA_2)$, then $N_{\cG_n}(x) \cap \ell_1 = N_{\cG_n}(x) \cap \ell_2 \in \{0, 1, q^2\}$, whereas if $x \in \cA_i$, then $N_{\cG_n}(x) \cap \ell_i = q^2$ and $N_{\cG_n}(x) \cap \ell_j = 0$, where $\{i, j\} = \{1, 2\}$.  

\begin{theorem}
The graphs $\cG'_n$ and $\cG''_n$ are strongly regular and have the same parameters as $\cG_n$.
\end{theorem}
\begin{proof}
It is sufficient to show that the hypotheses of Theorem~\ref{thm:wqh} are satisfied. Consider the graph $\cG'_n$. The induced subgraph on $\ell_i$ is the complete graph on $q^2$ vertices; as there is no edge between a vertex of $\ell_1$ and a vertex of $\ell_2$, the induced subgraph on $\ell_1 \cup \ell_2$ is the union of the two complete graphs. Let $u$ be a vertex of $\cG_n$ with $u \notin \ell_1 \cup \ell_2$. If $u$ is not in $\pi$, let $\sigma_i$ be the span of $u$ and $\ell_i$ for $i=1,2$. If $\sigma_i \cap \cH(n, q^2)$ is a line, then $N_{\cG_n}(u) \cap \ell_i = \ell_i$. If $\sigma_i \cap \cH(n, q^2)$ is a Hermitian pencil of lines, then all lines of $\sigma_i$ that are tangent to $\cH(n, q^2)$ go through $P$, so $|N_{\cG_n}(u) \cap \ell_i| = 0$. If $\sigma_i \cap \cH(n, q^2)$ is a non--degenerate Hermitian curve, then $\ell_i$ is the unique line of $\sigma_i$ that is tangent to $\cH(n, q^2)$ at $P$. In this case, as $u$ lies on $q+1$ lines of $\sigma_i$ that are tangent to $\cH(n, q^2)$, we have that $|N_{\cG_n}(u) \cap \ell_i| = q+1$. The curve $\sigma_1 \cap \cH(n, q^2)$ is non--degenerate if and only if the line $\sigma_1 \cap \sigma_2$ is secant to $\cH(n, q^2)$, i.e., if and only if the curve $\sigma_2 \cap \cH(n, q^2)$ is non--degenerate. Assume that $u \in \pi$. Thus $|N_{\cG_n}(u) \cap \ell_i| = 0$, $i = 1,2$. Hence, we have shown that we can apply Theorem~\ref{thm:wqh} and obtain a strongly regular graph with the same parameters as $\cG_n$. By using the same arguments used for $\cG'$, it is possible to see that the graph $\cG''_n$ is strongly regular and has the same parameters as $\cG_n$. 
\end{proof}

\subsection{The vertices adjacent to three pairwise adjacent vertices of $\cG_4$}

For a point $R$ of $\PG(4, q^2) \setminus \cH(4, q^2)$ let $\cT_{R}$ be the set of points lying on a line tangent to $\cH(4, q^2)$ at $R$. Thus $\cT_{R} = N_{\cG_4}(R) \cup (R^\perp \cap \cH(4, q^2)) \cup \{R\}$. Let $P_1, P_2, P_3$ be three points of $\PG(4, q^2) \setminus \cH(4, q^2)$ such that the line $P_i P_j$ is tangent to $\cH(4, q^2)$, $1 \le i < j \le 3$, or, in other words, $P_1, P_2, P_3$ are three pairwise adjacent vertices of $\cG_4$. Let $\gamma = \langle P_1, P_2, P_3 \rangle$. Note that if $\gamma$ is a plane, then $\gamma \cap \cH(4, q^2)$ is not a degenerate Hermitian curve of rank two. Indeed every tangent line contained in a plane meeting $\cH(4, q^2)$ in a degenerate Hermitian curve of rank two pass through the same point. Hence there are three possibilities:
\begin{itemize}
\item[1)] $\gamma$ is a line tangent to $\cH(4, q^2)$ and hence $\gamma \cap \cH(4, q^2)$ is a point;
\item[2)] $\gamma$ is a plane and $\gamma \cap \cH(4, q^2)$ is a line;
\item[3)] $\gamma$ is a plane and $\gamma \cap \cH(4, q^2)$ is a non--degenerate Hermitian curve.
\end{itemize}

\begin{lemma}\label{tanPlane}
Let $P_1, P_2, P_3$ be three points of $\PG(4, q^2) \setminus \cH(4, q^2)$ such that the line $P_i P_j$ is tangent to $\cH(4, q^2)$, $1 \le i < j \le 3$, $\gamma = \langle P_1, P_2, P_3 \rangle$ is a plane and $\gamma \cap \cH(4, q^2)$ is a line $\ell$.
\begin{itemize}
\item If $R$ is a point of $\PG(4, q^2) \setminus \left(\cH(4, q^2) \cup \gamma \right)$, then $\cT_{R} \cap \gamma$ is a Hermitian pencil of lines meeting $\ell$ in one point.
\item If $\cC$ is a Hermitian pencil of lines of $\gamma$ meeting $\ell$ in one point, then there are $q^3$ points $R \in \PG(4, q^2) \setminus \left(\cH(4, q^2) \cup \gamma \right)$ such that $\cT_{R} \cap \gamma = \cC$.
\end{itemize}
\end{lemma}
\begin{proof}
Let $R$ be a point of $\PG(4, q^2) \setminus \left(\cH(4, q^2) \cup \gamma \right)$. The plane $\langle R, \ell \rangle$ meets $\cH(4, q^2)$ in a Hermitian pencil of lines and hence the line $t = R^\perp \cap \gamma = \langle R, \gamma^\perp \rangle^\perp = \langle R, \ell \rangle^\perp$ is tangent to $\cH(4, q^2)$. By projecting the plane $\gamma$ from $R$ onto $R^\perp$, we get a plane $\gamma'$, where $\gamma \cap \gamma' = t$. Moreover, since $R^\perp \cap \cH(4, q^2)$ is a non--degenerate Hermitian surface and $\gamma' \subset R^\perp$, it follows that $|\gamma' \cap \cH(4, q^2)| \in \{q^3+q^2+1, q^3+1\}$. Note that if $L$ is a point of $\ell$, then the line $R L$ shares with $\cH(4, q^2)$ either one or $q+1$ points according as $L \in R^\perp$ or $L \notin R^\perp$, respectively, and $\ell \cap R^\perp = \ell \cap t$, since $\ell \subset \gamma$. Hence $|\ell \cap t| = 1$. By projecting the line $\ell$ from $R$ onto $R^\perp$, we obtain a line $\ell'$ that is tangent to $\cH(4, q^2)$, with $\ell' \subset \gamma'$ and $\ell' \cap \cH(4, q^2) = t \cap \cH(4, q^2)$. In particular in $\gamma'$, there are two distinct tangent lines, namely $t$ and $\ell'$, such that $t \cap \ell' \in \ell$ and hence $\gamma' \cap \cH(4, q^2)$ is a Hermitian pencil of lines Therefore, we deduce that there are $q^3+q^2$ points of $\gamma \setminus \cH(4, q^2)$ belonging to $N_{\cG_4}(R)$. This shows the first part of the claim.

In the plane $\gamma$, there are $q^2(q-1)(q^2+1)$ Hermitian pencils of lines meeting $\ell$ in one point, and these are permuted in a single orbit by the group $\PGU(5, q^2)_\gamma$, the stabilizer of $\gamma$ in $\PGU(5, q^2)$. On the other hand, $\PGU(5, q^2)_\gamma$ permutes in a unique orbit the $q^5(q-1)(q^2+1)$ points of $\PG(4, q^2) \setminus \left( \cH(4, q^2) \cup \gamma \right)$. Consider the incidence structure whose point set $\cP$ is formed by the points of $\PG(4, q^2) \setminus \left( \cH(4, q^2) \cup \gamma \right)$ and whose block set $\cL$ consists of the $q^2(q-1)(q^2+1)$ degenerate Hermitian curves described above, where a point $R \in \cP$ is incident with a block $\cC \in \cL$ if $\cT_R \cap \gamma = \cC$. Since both, $\cP$ and $\cL$, are orbits of the same group, the considered incidence structure is a tactical configuration and hence we have that a block of $\cL$ is incident with $q^3$ points of $\cP$, as required.
\end{proof}

\begin{lemma}\label{secPlane}
Let $P_1, P_2, P_3$ be three points of $\PG(4, q^2) \setminus \cH(4, q^2)$ such that the line $P_i P_j$ is tangent to $\cH(4, q^2)$, $1 \le i < j \le 3$, $\gamma = \langle P_1, P_2, P_3 \rangle$ is a plane and $\gamma \cap \cH(4, q^2)$ is a non--degenerate Hermitian curve.
\begin{itemize}
\item If $R$ is a point of $\gamma^\perp \setminus \cH(4, q^2)$, then $\cT_{R} \cap \gamma \subset \cH(4, q^2)$.
\item If $R$ is a point of $\PG(4, q^2) \setminus \left(\cH(4, q^2) \cup \gamma \cup \gamma^\perp \right)$ and $\langle R, \gamma \rangle \cap \gamma^\perp \in \cH(4, q^2)$, then $\cT_{R} \cap \gamma$ is a Hermitian pencil of lines, say $\cC$, meeting $\cH(4, q^2)$ in $q+1$ points. In this case there are $(q+1)(q^2-1)$ points $R \in \PG(4, q^2) \setminus \left(\cH(4, q^2) \cup \gamma \cup \gamma^\perp \right)$ such that $\cT_{R} \cap \gamma = \cC$.
\item If $R$ is a point of $\PG(4, q^2) \setminus \left(\cH(4, q^2) \cup \gamma \cup \gamma^\perp \right)$ and $\langle R, \gamma \rangle \cap \gamma^\perp \notin \cH(4, q^2)$, then $\cT_{R} \cap \gamma$ is a non--degenerate Hermitian curve, say $\cC$, meeting $\cH(4, q^2)$ in $1$ or $q+1$ points. In this case there are $(q+1)(q^2-q)$ points $R \in \PG(4, q^2) \setminus \left(\cH(4, q^2) \cup \gamma \cup \gamma^\perp \right)$ such that $\cT_{R} \cap \gamma = \cC$.
\end{itemize}
\end{lemma}
\begin{proof}
Let $R$ be a point of $\gamma^\perp \setminus \cH(4, q^2)$, then a line $t$ through $R$ that is tangent to $\cH(4, q^2)$ meets $\gamma$ in a point of $\cH(4, q^2)$. Hence $\cT_{R} \cap \gamma \subset \cH(4, q^2)$.

Let $R$ be a point of $\PG(4, q^2) \setminus \left( \cH(4, q^2) \cup \gamma \cup \gamma^\perp \right)$. The solid $\Gamma = \langle R, \gamma \rangle$ meets $\gamma^\perp$ in a point, say $R'$, and let $R'' = \Gamma^\perp = R^\perp \cap \gamma^\perp$. Let us define $\gamma'$ as the plane $R^\perp \cap \Gamma$ and $r$ as the line $\gamma \cap \gamma'$. Note that $\gamma'$ is the plane obtained by projecting $\gamma$ from $R$ onto $R^\perp$.

If $R' \in \cH(4, q^2)$, then $R' = R''$, $R'^\perp = \Gamma$, the line $RR'$ is tangent to $\cH(4, q^2)$ at $R'$ and the plane $\langle \gamma^\perp, R R' \rangle$ meets $\cH(4, q^2)$ in $q^3+1$ points. Hence $\gamma' = R^\perp \cap \Gamma = R^\perp \cap R'^\perp = \langle R R' \rangle^\perp$ is a plane meeting $\cH(4, q^2)$ in $q^3+q^2+1$ points. Moreover, $r = \langle \gamma^\perp, R R' \rangle^\perp$ has $q+1$ points in common with $\cH(4, q^2)$, since $|\langle \gamma^\perp, R R' \rangle \cap \cH(4, q^2)| = q^3+1$. In this case by projecting the $q^3+q^2+1$ points of $\gamma' \cap \cH(4, q^2)$ from $R$ onto $\gamma$, we get a Hermitian pencil of lines of $\gamma$, say $\cC$, such that $\cC \cap r = \cC \cap \cH(4, q^2) = r \cap \cH(4, q^2)$. Further the vertex of $\cC$ is the point $V = RR' \cap \gamma$. Varying the point $R'$ in $\gamma^\perp \cap \cH(4, q^2)$ and by replacing $R$ with one of the $q^2-1$ points of $V R' \setminus \left\{ V, R' \right\}$, we obtain the same curve $\cC$. 

If $R' \notin \cH(4, q^2)$, then $R' \ne R''$, $R''^\perp = \Gamma$ and the line $RR''$ is secant to $\cH(4, q^2)$, since $R \in R''^\perp$, with $R \notin \cH(4, q^2)$. Hence $\gamma' = R^\perp \cap \Gamma = R^\perp \cap R''^\perp = \langle R R'' \rangle^\perp$ is a plane meeting $\cH(4, q^2)$ in $q^3+1$ points. Note that $r = \gamma \cap \gamma' = \langle R', R'' \rangle^\perp \cap \langle R, R'' \rangle^\perp = \langle R, R', R'' \rangle^\perp$. Two possibilities arise: either $R R'$ is a line tangent to $\cH(4, q^2)$ at the point $RR' \cap \gamma$ or $RR'$ is secant to $\cH(4,q^2)$ and $RR' \cap \gamma \notin \cH(4, q^2)$. In the former case $|r \cap \cH(4, q^2)| = 1$, since $|\langle R, R', R'' \rangle \cap \cH(4, q^2)| = q^3+q^2+1$, whereas in the latter case $|r \cap \cH(4, q^2)| = q+1$, since $|\langle R, R', R'' \rangle \cap \cH(4, q^2)| = q^3+1$. By projecting the $q^3+1$ points of $\gamma' \cap \cH(4, q^2)$ from $R$ onto $\gamma$, we get a non--degenerate Hermitian curve of $\gamma$, say $\cC$, such that $\cC \cap r = \cC \cap \cH(4, q^2) = r \cap \cH(4, q^2)$. By Lemma \ref{HermSur} there are exactly $q+1$ points of $R R' \setminus \left( \cH(4, q^2) \cup \gamma \cup \gamma' \right)$ which give rise to the same curve $\cC$. Since $|\gamma^\perp \setminus \cH(4, q^2)| = q^2-q$, it follows that there are $(q+1)(q^2-q)$ points $R \in \PG(4, q^2) \setminus \left(\cH(4, q^2) \cup \gamma \cup \gamma^\perp \right)$ such that $\cT_{R} \cap \gamma = \cC$.
\end{proof}

\begin{theorem}\label{char}
Let $P_1, P_2, P_3$ be three pairwise adjacent vertices of $\cG_4$ and let $\gamma = \langle P_1, P_2, P_3 \rangle$.
\begin{enumerate}
\item If $\gamma$ is a line tangent to $\cH(4, q^2)$, let $s$ be the unique Baer subline of $\gamma$ containing $P_1, P_2, P_3$ and let $T = \gamma \cap \cH(4, q^2)$. Then
$$
|N_{\cG_4}(P_1) \cap N_{\cG_4}(P_2) \cap N_{\cG_4}(P_3)| =
\begin{cases}
q^5+q^4-q^3-3 & \mbox{ if } T \in s, \\
2q^5+q^4-q^3-3 & \mbox{ if } T \notin s.
\end{cases}
$$
\item If $\gamma$ is a plane and $\gamma \cap \cH(4, q^2)$ is a line, then $|N_{\cG_4}(P_1) \cap N_{\cG_4}(P_2) \cap N_{\cG_4}(P_3)| = q^5+3q^4-3$.
\item If $\gamma$ is a plane and $\gamma \cap \cH(4, q^2)$ is a non--degenerate Hermitian curve, then $|N_{\cG_4}(P_1) \cap N_{\cG_4}(P_2) \cap N_{\cG_4}(P_3)| = q^5+2q^4+3q^3-2q^2-q-3$.
\end{enumerate}
\end{theorem}
\begin{proof}
Assume that $\gamma$ is a line tangent to $\cH(4, q^2)$ and let $\sigma$ be a plane through $\gamma$. If $\sigma \cap \cH(4, q^2)$ is a line, then $\sigma$ contains $q^4-q^2$ points $R$ not in $\gamma$ such that the lines $R P_i$, $i = 1,2,3$, are tangent to $\cH(4, q^2)$. If $\sigma \cap \cH(4, q^2)$ is Hermitian pencil of lines, then $\sigma$ contains no point $R$ not in $\gamma$ such that the lines $R P_i$, $i = 1,2,3$, are tangent to $\cH(4, q^2)$, whereas if $\sigma \cap \cH(4, q^2)$ is a non--degenerate Hermitian curve, then from Lemma \ref{HermCurve1}, $\sigma$ contains either $0$ or $q$ points $R$ not in $\gamma$ such that the lines $R P_i$, $i = 1,2,3$, are tangent to $\cH(4, q^2)$ according as $T  \in s$ or $T \notin s$, respectively. Since there are $q+1$ planes meeting $\cH(4, q^2)$ in a line, $q^2-q$ planes intersecting $\cH(4, q^2)$ in a Hermitian pencil of lines and $q^4$ planes having $q^3+1$ points in common with $\cH(4, q^2)$, we have that $|N_{\cG_4}(P_1) \cap N_{\cG_4}(P_2) \cap N_{\cG_4}(P_3)| = (q+1)(q^4-q^2) + |\gamma \setminus \{ T, P_1, P_2, P_3\}| = q^5+q^4-q^3-3$, if $T \in s$, and $|N_{\cG_4}(P_1) \cap N_{\cG_4}(P_2) \cap N_{\cG_4}(P_3)| = (q+1)(q^4-q^2) + q^4 \cdot q + |\gamma \setminus \{ T, P_1, P_2, P_3\}| = 2q^5+q^4-q^3-3$, if $T \notin s$.

Assume that $\gamma$ is a plane and that $\gamma \cap \cH(4, q^2)$ is a line, say $\ell$. If $R$ is a point of $\PG(4, q^2) \setminus \left( \cH(4, q^2) \cup \gamma \right)$ such that the lines $R P_i$, $i = 1,2,3$, are tangent to $\cH(4, q^2)$, then from Lemma~\ref{tanPlane}, $\cT_{R} \cap \gamma$ is a Hermitian pencil of lines meeting $\ell$ in one point. From Lemma~\ref{HermCurve2}, in $\gamma$ there are $q^2+2q$ Hermitian pencils of lines meeting $\ell$ in one point and containing $P_1, P_2, P_3$. From Lemma \ref{tanPlane}, for each of these $q^2+2q$ Hermitian pencils of lines, there are $q^3$ points $R$ of $\PG(4, q^2) \setminus \left(\cH(4, q^2) \cup \gamma\right)$ such that the lines $R P_i$, $i = 1,2,3$, are tangent to $\cH(4, q^2)$. Therefore in this case $|N_{\cG_4}(P_1) \cap N_{\cG_4}(P_2) \cap N_{\cG_4}(P_3)| = q^3(q^2+2q) + |\gamma \setminus \left(\ell \cup \{P_1, P_2, P_3\}\right)| = q^5+3q^4-3$.

Assume that $\gamma$ is a plane and that $\gamma \cap \cH(4, q^2)$ is a non--degenerate Hermitian curve. If $R$ is a point of $\PG(4, q^2) \setminus \left( \cH(4, q^2) \cup \gamma \right)$ such that the lines $R P_i$, $i = 1,2,3$, are tangent to $\cH(4, q^2)$, then from Lemma~\ref{secPlane}, we have that $R \in \PG(4, q^2) \setminus \left( \cH(4, q^2) \cup \gamma \cup \gamma^\perp \right)$ and either $\cT_{R} \cap \gamma$ is a Hermitian pencil of lines meeting $\cH(4, q^2)$ in $q+1$ points or $\cT_{R} \cap \gamma$ is a non--degenerate Hermitian curve meeting $\cH(4, q^2)$ in one or $q+1$ points. From Lemma \ref{HermCurve3}, in $\gamma$ there are $3q$ Hermitian pencils of lines meeting $\gamma \cap \cH(4, q^2)$ in $q+1$ points and containing $P_1, P_2, P_3$. Similarly from Proposition \ref{HermCurve5}, in $\gamma$ there are $q^2-q+1$ non--degenerate Hermitian curves meeting $\gamma \cap \cH(4, q^2)$ in one or $q+1$ points and containing $P_1, P_2, P_3$. From Lemma \ref{secPlane}, for each of these Hermitian curves, there are either $(q+1)(q^2-1)$ or $(q+1)(q^2-q)$ points $R$ of $\PG(4, q^2) \setminus \left(\cH(4, q^2) \cup \gamma \cup \gamma^\perp \right)$ such that the lines $R P_i$, $i = 1,2,3$ are tangent to $\cH(4, q^2)$ according as the Hermitian curve is degenerate or not, respectively. Observe that if $R$ belongs to $\gamma$ then necessarily $R$ has to lie on a line joining two of the three points $P_1, P_2, P_3$, otherwise $\gamma \cap \cH(4, q^2)$ would contain the dual of an O'Nan configuration, a contradiction. Therefore in this case $|N_{\cG_4}(P_1) \cap N_{\cG_4}(P_2) \cap N_{\cG_4}(P_3)| = 3q \cdot (q+1)(q^2-1) + (q^2-q+1) \cdot (q+1)(q^2-q) + 3 \cdot (q-1) = q^5+2q^4+3q^3-2q^2-q-3$.
\end{proof}

\subsection{The isomorphism issue}

Set $n = 4$. Let $u$ be a point of $\ell_1$ and let $t$ be a line through $u$ tangent to $\cH(4, q^2)$ at $T$ and contained in $P^\perp$, with $t \ne \bar{\ell}_1$. If $\pi \cap \cH(4, q^2)$ is a Hermitian pencil of lines, then $|t \cap \cA_1| = q^2 - q - 2$, $|t \cap \cA_2| = 0$ and $t$ shares with $\cA$ the $q+1$ points of a Baer subline. If $\pi \cap \cH(4, q^2)$ is a line, then $|t \cap \cA_1| = q^2 - q - 1$, $|t \cap \cA_2| = 0$, $|t \cap \cA| = q$ and $(t \cap \cA) \cup \{u\}$ is Baer subline. Let $u_1$ and $u_2$ be two distinct points of $\cA$ and let $b$ be the Baer subline of $t$ containing $u, u_1, u_2$. In the case when $|\pi \cap \cH(4, q^2)| = q^3+q^2+1$, assume that $T \notin b$. Observe that $u, u_1, u_2$ are three pairwise adjacent vertices of $\cG_4$, $\cG'_4$ and $\cG_4''$.

\begin{prop}\label{char_new}
$|N_{\cG'_4}(u) \cap N_{\cG'_4}(u_1) \cap N_{\cG'_4}(u_2)| = |N_{\cG''_4}(u) \cap N_{\cG''_4}(u_1) \cap N_{\cG''_4}(u_2)| = 2q^5+q^3-3$.
\end{prop}
\begin{proof}
From the definition of $\cG'_4$ and $\cG''_4$ we have that
\begin{align*}
N_{\cG_4'}(u) \cap N_{\cG_4'}(u_1) \cap N_{\cG_4'}(u_2) & =  N_{\cG_4''}(u) \cap N_{\cG_4''}(u_1) \cap N_{\cG_4''}(u_2) \\
& = \left( \left( N_{\cG_4}(u) \setminus \cA_1 \right) \cup \cA_2 \right) \cap N_{\cG_4}(u_1) \cap N_{\cG_4}(u_2) \\
& =  \left( \left( N_{\cG_4}(u) \cap N_{\cG_4}(u_1) \cap N_{\cG_4}(u_2) \right) \setminus \cA_1 \right) \cup \left( N_{\cG_4}(u_1) \cap N_{\cG_4}(u_2) \cap \cA_2 \right) .
\end{align*}

Let $\sigma$ be the plane spanned by $t$ and $P$. Since $\sigma \cap \cH(4, q^2)$ is a line, it follows that every point of $\sigma \cap \cA_1$ spans with both $u_1$ and $u_2$ a line that is tangent to $\cH(4, q^2)$. Hence $\cA_1 \cap \sigma \subseteq N_{\cG_4}(u) \cap N_{\cG_4}(u_1) \cap N_{\cG_4}(u_2)$. Note that $|\cA_2 \cap \sigma| = 0$, whereas $|\cA_1 \cap \sigma|$ equals $q^2 (q^2-q-2)$ or $q^2(q^2-q-1)$ according as $\pi \cap \cH(4, q^2)$ is a Hermitian pencil of lines or a line. 

Denote by $\sigma'$ a plane through $t$ contained in $P^\perp$, with $\sigma' \ne \sigma$, and let $u' = \sigma' \cap \ell_2$. Then $\sigma' \cap \cH(4, q^2)$ is a non--degenerate Hermitian curve. 

If $|\pi \cap \cH(4, q^2)|$ equals $q^3+q^2+1$ (resp. $q^2+1$), the points of $\cA_1$ contained in $\sigma' \setminus \sigma$ are on the $q$ (resp. $q-1$) lines of $\sigma'$ tangent to $\sigma' \cap \cH(4, q^2)$ passing through $u$ and distinct from $t$ (resp. $t$ and $\sigma' \cap \pi$). By Lemma~\ref{HermCurve0}, there is a unique point $R$ on each of these tangent lines such that $R \in N_{\cG_4}(u) \cap N_{\cG_4}(u_1) \cap N_{\cG_4}(u_2)$ and if $R \in N_{\cG_4}(u) \cap N_{\cG_4}(u_1) \cap N_{\cG_4}(u_2)$, then $R \in (\cA \cup \cA_1) \cap (\sigma' \setminus \sigma)$. In particular, if the case $\langle R, u' \rangle \cap t \notin \{u_1, u_2\}$ occurs, then $R \notin \cA$, otherwise $R, u_1, u_2, u'$ would be a dual O'Nan configuration of $\sigma' \cap \cH(4, q^2)$, which is impossible. Hence $R \in \cA_1$. Assume that $\langle R, u' \rangle \cap t \in \{u_1, u_2\}$. If $\pi \cap \cH(4, q^2)$ is a Hermitian pencil of lines, then $R \in \cA$, by Lemma~\ref{HermCurve-1}. If $\pi \cap \cH(4, q^2)$ is a line, then $R \in \cA_1$, otherwise $R, u_1, u, u'$ or $R, u_2, u, u'$ would be a dual O'Nan configuration of $\sigma' \cap \cH(4, q^2)$. Therefore $|N_{\cG_4}(u) \cap N_{\cG_4}(u_1) \cap N_{\cG_4}(u_2) \cap \cA_1|$ equals $q^2(q^2-q-2)+q^2(q-2) = q^4-4q^2$ or $q^2(q^2-q-1)+q^2(q-1) = q^4-2q^2$, according as $\pi \cap \cH(4, q^2)$ is a Hermitian pencil of lines or a line, respectively. 

In a similar way, if $|\pi \cap \cH(4, q^2)|$ equals $q^3+q^2+1$ (resp. $q^2+1$), the points of $\cA_2$ contained in $\sigma' \setminus \sigma$ are on the $q+1$ (resp. $q$) lines of $\sigma'$ tangent to $\sigma' \cap \cH(4, q^2)$ passing through $u'$ (and distinct from $\sigma' \cap \pi$). Each of these lines meets $t$ in a point of $\cA$. Let $R$ be a point on one of these $q+1$ (resp. $q$) lines, with $R \in N_{\cG_4}(u_1) \cap N_{\cG_4}(u_2) \setminus (t \cup \{u'\})$. Hence $R \in (\cA \cup \cA_2) \cap (\sigma' \setminus \sigma)$. If $\langle R, u' \rangle \cap t \notin \{u_1, u_2\}$, then $R \notin \cA_2$, otherwise $R, u_1, u_2, u'$ would be a dual O'Nan configuration of $\sigma' \cap \cH(4, q^2)$, a contradiction. Assume that $\langle R, u' \rangle \cap t \in \{u_1, u_2\}$. There are $q-1$ possibilities for $R$, since the line $\langle u', u_i \rangle$ contains $q-1$ points distinct from $u'$ belonging to $N_{\cG_4}(u_1) \cap N_{\cG_4}(u_2)$, by Lemma \ref{HermCurve0}. If $\pi \cap \cH(4, q^2)$ is a Hermitian pencil of lines, then exactly one of these $q-1$ points is in $\cA$ by Lemma \ref{HermCurve-1}. If $\pi \cap \cH(4, q^2)$ is a line, then for all the $q-1$ possibilities we have that $R \in \cA_2$, otherwise $R, u_1, u, u'$ or $R, u_2, u, u'$ would be a dual O'Nan configuration of $\sigma' \cap \cH(4, q^2)$. Therefore $|N_{\cG_4}(u_1) \cap N_{\cG_4}(u_2) \cap \cA_2 \cap \sigma'|$ is $2(q-2)$ or $2(q-1)$ and hence $|N_{\cG_4}(u_1) \cap N_{\cG_4}(u_2) \cap \cA_2|$ equals $2(q-2) q^2$ or $2(q-1)q^2$, according as $\pi \cap \cH(4, q^2)$ is a Hermitian pencil of lines or a line, respectively. 

It follows that $|N_{\cG_4'}(u) \cap N_{\cG_4'}(u_1) \cap N_{\cG_4'}(u_2)| = 2q^5+q^4-q^3-3 - (q^4-4q^2) + 2(q^3 - 2q^2) = 2q^5+q^3-3$ and $|N_{\cG_4''}(u) \cap N_{\cG_4''}(u_1) \cap N_{\cG_4''}(u_2)| = 2q^5+q^4-q^3-3 - (q^4-2q^2) + 2(q^3 - q^2) = 2q^5+q^3-3$.
\end{proof}

As a consequence of Theorem \ref{char} and Proposition \ref{char_new}, we have the following.

\begin{theorem}
Let $n \ge 4$. The graph $\cG_n$ is not isomorphic neither to $\cG_n'$ (if $q >2$), nor to $\cG''_n$.
\end{theorem}
\begin{proof}
Let $\Pi$ be a four--space of $\PG(n, q)$ such that $\Pi \cap \cH(n, q^2) = \cH(4, q^2)$ and $\bar{\ell}_i \subset \Pi$, $i = 1, 2$. Then the induced subgraph by $\cG_n$ or $\cG_n'$ or $\cG_n''$ on $\Pi \setminus \cH(4, q^2)$ is $\cG_4$ or $\cG_4'$ or $\cG_4''$. Let $u, u_1, u_2$ be three pairwise adjacent vertices of $\cG_4$, $\cG_4'$ and $\cG_4''$. By Theorem \ref{char}, $|N_{\cG_4}(u) \cap N_{\cG_4}(u_1) \cap N_{\cG_4}(u_2)|$ takes four different values. By Proposition \ref{char_new}, both $|N_{\cG_4'}(u) \cap N_{\cG_4'}(u_1) \cap N_{\cG_4'}(u_2)|$ and $|N_{\cG_4''}(u) \cap N_{\cG_4''}(u_1) \cap N_{\cG_4''}(u_2) |$ take at least five different values.
\end{proof}

\section{Conclusions}

We have seen that if $n \ne 3$, then $\NU(n+1, q^2)$ is not determined by its spectrum. Regarding the case $n = 3$, there are exactly $28$ strongly regular graphs cospectral with $\NU(4, 4)$, see \cite{Spence}. A strongly regular graph having the same parameters as $\NU(4, 9)$ and admitting $2 \times {\rm U}(4, 4)$ as a full automorphism group has been constructed in \cite{CRS}. We ask whether or not the graph $\NU(4, q^2)$, $q \ge 4$, is determined by its spectrum.

\bigskip

\noindent\textit{Acknowledgments.} The first author is supported by a postdoctoral fellowship of the Research Foundation - Flanders (FWO).

\end{document}